\documentclass[a4paper,leqno,DIV=10,11pt,headings=normal,oneside]{scrartcl}
\pdfoutput=1 
\usepackage[utf8]{inputenc}
\usepackage[english]{babel}
\usepackage[autostyle=true]{csquotes}
\usepackage[T1]{fontenc}
\usepackage{microtype}
\usepackage{kpfonts}

\usepackage[colorlinks=false,pdftitle={},pdfauthor={Rudolf Zeidler},pdfborderstyle={/S/U/W 1},linkbordercolor={0 0 1},hyperfootnotes=false]{hyperref}
\usepackage{amsmath}
\usepackage{amssymb}
\usepackage{amsthm}
\usepackage{thmtools,thm-restate}

\allowdisplaybreaks[1]

\setkomafont{disposition}{\normalcolor\bfseries}
\usepackage[nouppercase,headsepline]{scrpage2}

\makeatletter
\def\blfootnote{\gdef\@thefnmark{}\@footnotetext}
\makeatother

\usepackage{tikz}
\usetikzlibrary{matrix, arrows}
\usepackage{graphicx}

\usepackage[citestyle=alphabetic,bibstyle=my-alphabetic,backend=biber,url=false,doi=true,isbn=false,maxbibnames=99,firstinits=true,dateabbrev=false]{biblatex}
\renewcommand*{\bibfont}{\small}

\bibliography{literature.bib}

\usepackage[noabbrev]{cleveref}

\usepackage{paralist}
\usepackage{color}

\crefname{equation}{formula}{Formula}
\crefname{enumi}{}{}
\crefname{figure}{Figure}{Figure}
\crefname{table}{Table}{Table}

\declaretheorem[name=Theorem, numberwithin=section]{thm}
\declaretheorem[name=Theorem, numbered=no]{thm*}
\declaretheorem[name=Lemma,numberlike=thm]{lem}
\declaretheorem[name=Corollary,numberlike=thm]{cor}
\declaretheorem[name=Corollary,numbered=no]{cor*}
\declaretheorem[name=Proposition,numberlike=thm]{prop}
\declaretheorem[name=Definition,numberlike=thm, style=definition]{defi}
\declaretheorem[name=Example, numberlike=thm, style=definition]{ex}
\declaretheorem[name=Remark, numberlike=thm, style=definition]{rem}

\crefname{thm}{Theorem}{Theorems}
\crefname{lem}{Lemma}{Lemmas}
\crefname{defi}{Definition}{Definitions}
\crefname{cor}{Corollary}{Corollaries}
\crefname{prop}{Proposition}{Propositions}
\crefname{ex}{Example}{Examples}
\crefname{rem}{Remark}{Remarks}
\crefname{section}{Section}{Sections}
\crefname{chapter}{Chapter}{Chapters}

\creflabelformat{enumi}{(#2#1#3)}

\newenvironment{definlist}
{\begin{inparaenum}[(i)]}
{\end{inparaenum}}
\newenvironment{thminlist}
{\begin{inparaenum}[(i)]}
{\end{inparaenum}}
\numberwithin{equation}{section}

\newcommand{\Lp}{\mathrm{L}}
\newcommand{\field}[1]{\mathbb{#1}}
\newcommand{\C}{\field{C}}
\newcommand{\Z}{\field{Z}}
\newcommand{\N}{\field{N}}
\newcommand{\Q}{\field{Q}}
\newcommand{\R}{\field{R}}
\newcommand{\Rgeq}{\R_{\geq 0}}
\newcommand{\Rgr}{\R_{> 0}}
\newcommand{\euclSpace}{\mathbb{E}}
\newcommand{\proj}{\operatorname{pr}}

\newcommand{\ev}{\operatorname{ev}}
\newcommand{\walls}{\mathcal{W}}
\newcommand{\cat}{\mathrm{CAT}}
\newcommand{\Intv}{\mathrm{I}}
\newcommand{\rk}{\operatorname{rank}}

\newcommand{\hypCM}{\mathrm{hyp}}

\newcommand{\isom}{\operatorname{Isom}}
\newcommand{\qisom}{\operatorname{QIsom}}

\newcommand{\relmiddle}[1]{\mathrel{}\middle#1\mathrel{}}
\newcommand{\middlemid}{\relmiddle{|}}

\newcommand{\assympCone}{\mathrm{Cone}}
\newcommand{\basePt}{\bullet}
\newcommand{\fixedPtProperty}{\mathrm{F}}
\newcommand{\FHomCone}{\fixedPtProperty_{\assympCone}}

\newcommand{\coarseEq}{\sim}

\renewcommand{\geq}{\geqslant}
\renewcommand{\leq}{\leqslant}

\newcommand{\finMedSub}{\mathscr{M}}
\newcommand{\median}{\mu}

\newcommand{\cube}{\mathscr{C}}

\newcommand{\myAffiliation}{{
 \textsc{Mathematisches Institut, Georg--August--Universität Göttingen, Bunsenstraße 3-5, 37073 Göttingen, Germany}
  
  \textit{E-mail address:} \href{mailto:rzeidle@uni-goettingen.de}{\nolinkurl{rzeidle@uni-goettingen.de}}
}}

\title{Coarse median structures and homomorphisms from Kazhdan groups\blfootnote{\textup{2010} \textit{Mathematics Subject Classification.} 20F65.}}
\author{Rudolf Zeidler\thanks{The research was conducted mostly while the author prepared his Master~thesis~\cite{zeidler:masterThesis} at the University of Vienna and was partially supported by the European Research Council (ERC) grant of Goulnara Arzhantseva, grant agreement n\textdegree259527.
The author is currently supported by the German Research Foundation (DFG) through the Research Training Group 1493 \enquote{Mathematical structures in modern quantum physics.}}}
\date{}

\begin{document}

\maketitle
\begin{abstract}
 We study Bowditch's notion of a coarse median on a metric space and formally introduce the concept of a \emph{coarse median structure} as an equivalence class of coarse medians up to closeness.
 We show that a group which possesses a uniformly left-invariant coarse median structure admits only finitely many conjugacy classes of homomorphisms from a given group with Kazhdan's property (T).
 This is a common generalization of a theorem due to Paulin about the outer automorphism group of a hyperbolic group with property (T) as well as of a result of Behrstock--Dru{\c{t}}u--Sapir on the mapping class groups of orientable surfaces.
 We discuss a metric approximation property of finite subsets in coarse median spaces extending the classical result on approximation of Gromov hyperbolic spaces by trees.
\end{abstract}

\section{Introduction}
\emph{Coarse median spaces} may be viewed as a large-scale generalization of both median algebras and median metric spaces including the class of Gromov hyperbolic spaces.
The concept of a coarse median on a metric space has been first introduced and studied by Bowditch~\cite{bowditch:coarseMedian,bowditch:embeddingMedianAlgebrasIntoTrees,bowditch:RelHyperbolicityInvarianceOfcoarseMedian}.
A \emph{coarse median} is a ternary operation $\mu : X \times X \times X \to X$ on a metric space $X$ which --- in a certain precise sense --- satisfies the axioms of a median algebra up to bounded error.

Recapitulating~\cite{bowditch:coarseMedian}, the main features are as follows.
The existence of a coarse median is a quasi-isometry invariant of the underlying metric space. Thus the theory is well-suited to be applied in the study of finitely generated groups.
There is a notion of ``rank'' for coarse medians and the Gromov hyperbolic spaces are precisely those spaces which admit a coarse median of rank $1$.
More generally, asymptotic cones of spaces admitting a rank $n$ coarse median are metric median algebras of rank at most $n$.
In particular, a space admitting a coarse median of rank $n$ does not admit a quasi-isometric embedding of $\R^{m}$ for $m > n$.
Coarse median groups are always finitely presented with an at most quadratic Dehn function.
Easy higher rank examples such as $\Z^n$ can be obtained because direct products of coarse median spaces are coarse median themselves.
Moreover, based on the work of Behrstock--Minsky~\cite{behrstock-minsky:centroidsAndRapidDecay}, Bowditch has shown that the mapping class group of an orientable surface admits a coarse median of rank at most the complexity of the surface.
The class of coarse median spaces is also stable under relative hyperbolicity~\cite{bowditch:RelHyperbolicityInvarianceOfcoarseMedian}.

\subsection{Coarse median spaces} \label{subsec:introToCoarseMedian}
We begin with the definition of a coarse median.
As a prerequisite we need the concept of a \emph{finite median algebra}.
As advocated by Bowditch~\cite{bowditch:coarseMedian}, we think of a finite median algebra as the vertex set of a finite $\cat(0)$ cube complex together with the ternary operation sending each triple of vertices to their median point according to the median metric induced by the $1$-skeleton (cf.~\cite[\S 10]{roller:pocSets-MedianAlgebras}).
For finite median algebras this viewpoint is completely equivalent to the more traditional axiomatic definition.
See~\cite{bowditch:coarseMedian} for a discussion well suited to our purposes.
For a broader viewpoint on median algebras and related structures we refer to \cite{bandelt-hedlikova:medianAlgebras,roller:pocSets-MedianAlgebras,chepoi:medianGraphs}.
 
 Let $(X, d)$ be a metric space.
 A ternary operation $\mu \colon X^3 \to X$ is called a \emph{coarse median} if there is a constant $k \in \Rgeq$ and a non-decreasing function $h \colon \N \to \Rgeq$ such that the following conditions are satisfied.
 
 \begin{definlist}
  \item The map $\mu$ is \emph{coarsely Lipschitz} with multiplicative constant $k$ and additive constant $h(0)$, that is, \label{item:cM1}for every $x, y, z, x^\prime, y^\prime, z^\prime \in X$, we have,
  \begin{equation*} \label{eq:coarseMedianCont}
   d\left( \mu(x, y, z), \mu(x^\prime, y^\prime, z^\prime) \right) \leq k \left( d(x,x^\prime) + d(y, y^\prime) + d(z, z^\prime) \right) + h(0).
  \end{equation*}
  
  \item \label{item:cM2} For every finite non-empty subset $A \subseteq X$, there is a finite median algebra $(M, \mu_M)$ together with maps $\lambda \colon M \to X$ and $\pi \colon A \to M$ such that 
  \begin{align*}
   d(a, \lambda(\pi(a))) &\leq h(|A|) \qquad \text{for all $a \in A$}, \\
   d\left( \lambda(\mu_M(m_1,m_2,m_3)), \mu(\lambda(m_1), \lambda(m_2), \lambda(m_2)) \right) &\leq h(|A|)      \qquad \text{for all $m_i \in M, i \in \{1,2,3\}$}.
  \end{align*}
 \end{definlist}
 We call $k$ and $h$ \emph{parameters} of the coarse median $\mu$.
 Note that we only require their existence and we do not consider a particular choice of parameters to be part of the structure of a coarse median.
 However, if it is possible to fix parameters such that the cube complexes underlying all the finite median algebras in condition~\labelcref{item:cM2} have dimension at most $n$, then we say that the coarse median $\mu$ has \emph{rank at most $n$}.
 
 A \emph{coarse median space} is a metric space together with a coarse median.

 Intuitively, the inequalities in condition \labelcref{item:cM2} mean that $\lambda$ is almost a homomorphism of the respective ternary operations and $\lambda \circ \pi$ is almost the inclusion $A \hookrightarrow X$.
 In view of the previous discussion of finite median algebras, condition \labelcref{item:cM2} may be interpreted as saying that every finite subset $A \subseteq X$ is (in a certain sense) approximately contained in a finite $\cat(0)$ cube complex, up to an error which only depends on the cardinality of $A$.

\subsection{Statement of results}
In order to compare different coarse medians in a meaningful way, we introduce the concept of a \emph{coarse median structure} on a metric space as a closeness class of coarse medians.
Here we say that two coarse medians $\mu, \mu^\prime \colon X^3 \to X$ are \emph{close} if there exists $R \geq 0$ such that $d\left( \mu(x_1,x_2,x_3), \mu^\prime(x_1,x_2,x_3) \right) \leq R$ for all $x_1,x_2,x_3 \in X$.
This has the advantage that it is straightforward to construct a well-defined \emph{pullback} and \emph{pushforward} of a coarse median structure via a quasi-isometry (cf.~\cref{defi:inducedMedian}).
One verifies that this construction satisfies several formal properties (cf.~\cref{prop:formalPropOfInducedCM}). 
In particular, the pushforward and pullback are compatible with the composition of quasi-isometries.
We illustrate the usefulness of these formal definitions by considering hyperbolic groups acting on $\cat(0)$ cubical complexes.
Indeed, in such a case, the coarse median structure on the hyperbolic group is related to the natural coarse median structure on the cube complex via the quasi-isometry defined by the orbit map (cf.~\cref{ex:HyperbolicGpactingOnCCC}).

Moreover, we say that a coarse median structure on a finitely generated group is \emph{uniformly left-invariant} if one (and hence any) of its representatives is almost equivariant with respect to the left multiplication action up to a uniform error (cf.~\cref{defi:uniformLeftInvariant}).
The coarse median structure on a hyperbolic group is automatically uniformly left-invariant.
This leads us to our main results, \cref{thm:paulinForCoarseMedian,cor:paulinForCoarseMedian}, which generalize the work of Paulin~\cite{paulin:outerAutomorphismsOfHyperbolicGroups} asserting that a hyperbolic group with Kazhdan's property (T) has finite outer automorphism group.
\begin{restatable*}{thm}{thmPaulinForCoarseMedian} \label{thm:paulinForCoarseMedian}
 Let $G$ be a finitely generated group which admits a uniformly left-invariant coarse median structure of finite rank. 
 Then for any finitely generated group $H$ with Kazhdan's property (T), there are only finitely many conjugacy classes of homomorphisms from $H$ into $G$.
\end{restatable*}
\begin{restatable*}{cor}{corPaulinForCoarseMedian} \label{cor:paulinForCoarseMedian}
   If a group admits a uniformly left-invariant coarse median structure of finite rank and has Kazhdan's property (T), then its outer automorphism group is finite.
\end{restatable*}
For example, the theorem applies if $G = G_1 \times \dotsm \times G_n$ is a finite direct product of hyperbolic groups $G_i$.
In particular, by the corollary, if such a group $G$ has Kazhdan's property (T), then its outer automorphism group is finite (cf. \cref{ex:productsOfHyperbolicWithT}).

\cref{thm:paulinForCoarseMedian} is inspired by the analogous theorem for mapping class groups due to Behrstock--Dru{\c{t}}u--Sapir~\cite[Theorem 1.2]{behrstock-drutu-sapir:MedianStructuresOnAymptoticCones}.
In fact, using Bowditch's work on coarse medians on the mapping class groups~\cite[Theorem 2.5]{bowditch:coarseMedian}, their result is a special case of our theorem.

Our proof of \cref{thm:paulinForCoarseMedian} follows the same general idea as set forth by Paulin's original work (the main steps of which we recapitulate in \cref{subsec:genericPaulin}).
However, to make it work in our setting, the crucial ingredient is a new fixed point theorem for certain actions of groups with Kazhdan's property (T) on metric median algebras (cf.~\cref{thm:TfixedPoint}).

The proof of this fixed point result, which is of independent interest, proceeds as follows.
First, we show that the metric on an appropriate class of metric median algebras can be canonically bi-Lipschitz deformed so that the space becomes a median metric space (cf.~\cref{subsec:rectifying}).
Second, we use that fact that every isometric action of a Kazhdan group on a median space has bounded orbits~\cite{nica:masterThesis,chatterji_drutu_haglund:haagerupMedian}.
Third, we appeal to a recent construction due to Bowditch~\cite{bowditch:somePropertiesOfMedianSpaces} implying that our median metric space can be bi-Lipschitz deformed into a $\cat(0)$ space, where finally we are in a position to apply the well-known fixed point lemma for isometric actions with bounded orbits.

In a different direction, we discuss metric approximations of finite subsets in coarse median spaces by finite median metric spaces. 
This uses ideas similar in spirit to what is also needed for the deformation of metric median algebras into median spaces.
The definition of a coarse median already states that every finite subset can be approximated --- on the level of ternary operations --- by a finite median algebra.
We show that this is enough to imply the following metric approximation property.

\begin{restatable*}{thm}{thmFiniteSubsetApproxByCCC} \label{thm:finiteSubsetApproxByCCC}
 Let $(X,d)$ be a coarse median space. Then there exist functions $\alpha, \varepsilon \colon \N \to \Rgeq$ (which only depend on parameters of the coarse median on $X$) such that the following holds.
 
 For each finite subset $A \subseteq X$, there exists a finite median metric space $(M, d_l)$, as well as an $(\alpha(|A|), \varepsilon(|A|))$-quasi-isometric embedding $f \colon (M, d_l) \to (X,d)$, such that $A \subseteq f(M)$.
 Moreover, if $(X,d)$ admits a coarse median of rank at most $n$, we can assume $M$ to always have rank at most $n+1$.
\end{restatable*}

Since a finite median metric space is just the vertex set of a finite $\cat(0)$ cube complex (with possibly rescaled edge lengths), this result may be reframed as approximation by such finite $\cat(0)$ cubical complexes.
Hence it generalizes the classical result on approximating trees in Gromov hyperbolic spaces~\cite[\S 6.2]{gromov:hyperbolicGroups}.

\subsection{Plan of the paper}
The paper is organized as follows.

In \cref{sec:prerequ}, we review basic facts about median metric spaces and median algebras, and set up some notation for later use.

Our method of turning certain metric median algebras into honest median metric spaces and the conclusion of our fixed point result is explained in \cref{sec:metricMedian}.

In \cref{sec:qiInvariance}, we revisit the definition of coarse medians and introduce coarse median structures.
Then we proceed with our formal study of quasi-isometry invariance of coarse median spaces and finally discuss the example of hyperbolic groups acting on $\cat(0)$ cubical complexes.

\Cref{sec:homosFromKazhdan} is dedicated to the proof of our main result on homomorphisms from Kazhdan groups into groups with an uniformly left-invariant coarse median structure.

Finally, \Cref{sec:approxFiniteSubsets}, which is concerned with our approximation result about finite subsets in coarse median spaces, can be read mostly independently from the rest of the paper (after possibly reviewing some notation from \cref{sec:prerequ}). 

\subsection{Acknowledgements}
The author would like to express his gratitude towards his Master thesis advisor, Goulnara Arzhantseva, for introducing him to this subject, sharing her knowledge, many helpful suggestions, and most of all, for encouraging him to write this paper.
He wishes to thank Fr{\'e}d{\'e}ric Paulin for giving an enlightening mini-course at the second Young Geometric Group Theory Meeting which stimulated the author's work towards a generalization of the result concerning outer automorphisms of hyperbolic groups with Kazhdan's property (T).
The author also thanks Bogdan Nica and the anonymous referee for useful comments on the manuscript.

\section{Prerequisites} \label{sec:prerequ}
\subsection{Median metric spaces} \label{subsec:medianAlgebras}
In this and the following subsection, we collect standard results on median metric spaces and median algebras.
These and related structures have been studied in various contexts by many different authors, so we can only mention a selection of references \cite{bandelt-hedlikova:medianAlgebras,roller:pocSets-MedianAlgebras,chepoi:medianGraphs,nica:masterThesis,chatterji_niblo:wallToCat0,
chatterji_drutu_haglund:medianSpaces,chatterji_drutu_haglund:haagerupMedian} without claiming historical completeness.

Let $(X,d)$ be a metric space and $x, y \in X$. The \emph{metric interval} between $x$ and $y$ is the set $\Intv_X^d(x,y) = \left\{ z \in X \middlemid d(x,z) + d(z,y) = d(x,y) \right\}$.
\begin{defi} \label{defi:medianMetric}
 A metric space $(X,d)$ is a \emph{median metric space} if the set $\bigcap_{i=1}^3 \Intv_X^d(x_i,x_{i+1})$ (indices are taken modulo $3$) has exactly one element $\mu_d(x_1,x_2,x_3)$ for all $x_1,x_2,x_3 \in X$.
\end{defi}
Given a median metric space $(X,d)$, we view $\mu_d$ as a ternary operation $X \times X \times X \to X$ and call it the \emph{intrinsic median operation} of $(X,d)$.

\begin{ex}
 The vertex set $C^{(0)}$ of a $\cat(0)$ cubical complex $C$ endowed with the restriction of the path metric from the $1$-skeleton is a median metric space~\cite[\S 10]{roller:pocSets-MedianAlgebras}~\cite[Theorem 6.1]{chepoi:medianGraphs}.
 We denote the median operation on $C^{(0)}$ by $\mu_{C^{(0)}}$.
\end{ex}

\begin{ex}
 $\R^n$ endowed with the metric induced by the $1$-norm is a median metric space.
 More generally, $\Lp^1(X)$ is a median metric space for any measured space $X$.  
\end{ex}

\subsection{Median algebras}
\begin{defi} \label{defi:medianAlgebra}
 A \emph{median algebra} is a set $M$ together with a ternary operation $\mu \colon M^3 \to M$ such that for every $x, y, z, u, v \in M$, the following identities hold.
 \begin{align*}
  &\mu(x, y, z) = \mu(x, z, y) = \mu(y, z, x), \\
  &\mu(x, x, y) = x, \\
  &\mu(\mu(x,y,z), u, v) = \mu(x, \mu(y, u, v), \mu(z, u, v))
 \end{align*}
 \end{defi}
 Given a median algebra $(M, \mu)$ and $x, y \in M$, the \emph{median interval} between $x$ and $y$ is the set $\Intv_{M}^\mu(x,y) := \left\{ z \in M \mid \mu(x,y,z) = z \right\}$.

  A subset $A \subseteq M$ of a median algebra is called a \emph{subalgebra} if for any $a,b,c \in A$, their median $\mu(a,b,c)$ lies in $A$ as well.
 The smallest subalgebra of $M$ containing $A$ is called the \emph{subalgebra generated by $A$}, and we denote it by $\left<A\right>$.
The subset $A$ is called \emph{convex} if for all $a,b \in A$, we have $\Intv_{M}^\mu(a,b) \subseteq A$.

A \emph{wall} $\{h, h^c\}$ in a median algebra is a partition $M = h \sqcup h^c$ into two non-empty and convex subsets $h, h^c \subseteq M$.
We denote the set of walls in $(M,\mu)$ by $\walls(M,\mu)$ or just $\walls(M)$.
A wall $\{h,h^c\} \in \walls(M)$ \emph{separates} two points $a,b \in M$ if $a \in h$ and $b \in h^c$ or vice-versa.
The set of walls separating $a$ and $b$ is denoted by $\walls^M(a|b)$.
Given two distinct elements of a median algebra, there always exists a wall separating them.
In fact, any two disjoint convex subsets can be separated by a wall~\cite[\S 2]{roller:pocSets-MedianAlgebras}.

An \emph{$n$-cube} in a median algebra is a subalgebra isomorphic to the median algebra $\{-1,1\}^n$.

\begin{rem} \label{rem:medianMetricIsMedianAlgebra}
Every median metric space admits a unique structure as a median algebra such that the metric intervals agree with the median intervals.
This can easily be seen from the characterization of median algebras stated in~\cite[Theorem 2.1]{bandelt-hedlikova:medianAlgebras}.
The corresponding ternary operation is just the intrinsic median operation $\mu_d$ from \cref{defi:medianMetric}.
\end{rem}

\begin{ex}
 The vertex set of a $\cat(0)$ cubical complex is a median algebra in a natural way.
 The median walls as defined above coincide with the walls determined by the geometric hyperplanes of the $\cat(0)$ cubical complex.
\end{ex}

Two edges in a $\cat(0)$ cubical complex are said to be \emph{parallel} if they intersect the same wall.

\begin{rem}[{\cite{chatterji_niblo:wallToCat0},\cite[\S 10]{roller:pocSets-MedianAlgebras},\cite[Section 5]{bowditch:coarseMedian}}] 
 \label{finiteMedianAlgebraIsCube}
Conversely, if $(M,\mu)$ is a finite median algebra, then we can construct a (simplicial) graph $\cube^{(1)}(M)$ with vertex set $\cube^{(0)}(M,\median) := M$ and an edge between two elements $x,y \in M$ whenever there exists a \emph{unique} wall $W \in \walls(M,\median)$ separating $x$ and $y$.
The graph $\cube^{(1)}(M,\median)$ can be uniquely completed to a finite $\cat(0)$ cubical complex $\cube(M,\mu)$.
The path metric on $\cube^{(1)}(M)$ induces a median metric on $M$ with intrinsic median operation equal to $\median$.
In fact, there is a bijective correspodence $M \mapsto \cube(M)$ between isomorphism classes of finite median algebras and isomorphism classes of finite $\cat(0)$ cubical complexes.
\end{rem}

\begin{defi} \label{defi:rankOfMedianAlg}
 The \emph{rank} of a finite median algebra $M$ is the dimension of $\cube(M)$. The \emph{rank} of a general median algebra is the (possibly infinite) supremum over the ranks of all finite subalgebras.
\end{defi}

In \cref{defi:medianAlgebra}, we have given the abstract definition of a median algebra. 
However, in the spirit of~\cite{bowditch:coarseMedian}, we will not need to use this definition in practice.
This is because the median subalgebra generated by a finite subset of some median algebra is always finite itself.

\begin{lem}[{see e.g.~\cite{nica:masterThesis},\cite[Lemma 4.2]{bowditch:coarseMedian}}] \label{lem:finiteSubalgebra}
 For every finite subset $A \subseteq M$ of a median algebra, the cardinality of the subalgebra generated by $A$ is bounded above by $2^{2^{|A|}}$.
\end{lem}

Together with \cref{finiteMedianAlgebraIsCube} this implies a characterization of median algebras as follows.

\begin{cor} \label{cor:alternativeDefOfMedianAlgebra}
 Let $M$ be a set and $\mu \colon M^3 \to M$ a ternary operation. Then $(M, \mu)$ is a median algebra if and only if the following holds:
 
 For every finite subset $A^\prime \subseteq M$, there is another finite subset $A^\prime \subseteq A \subseteq M$ such that $A$ is closed under $\mu$ and $(A, \mu)$ is isomorphic to the median algebra determined by some finite $\cat(0)$ cubical complex.
\end{cor}

In other words, in order to understand median algebras it suffices to understand the median structures of finite $\cat(0)$ cubical complexes.

\subsection{Constructing metrics on finite median algebras} \label{subsec:constructingMetricsOnMedianAlgebras}
In this short subsection, we discuss how to describe median metrics on finite median algebras.
This will be needed in \cref{subsec:rectifying,sec:approxFiniteSubsets}.

We begin with a useful observation which appears in~\cite[Section 2]{bowditch:somePropertiesOfMedianSpaces}.
\begin{lem} \label{lem:medianCriterion}
 Let $(X, \mu)$ be a median algebra and $\rho \colon X \times X \to \R$ any metric on $X$.
 Suppose that for all $x,y,z \in X$ with $z = \mu(x,y,z)$, we have $\rho(x,y) = \rho(x,z) + \rho(z,y)$.
 Then $(X,\rho)$ is a median metric space with intrinsic median operation $\mu_\rho = \mu$.
\end{lem}

Now let $(M,\mu)$ be a finite median algebra and $l \colon \walls(M) \to \Rgr$ any map.
Then we define a distance function $d_l \colon M \times M \to \Rgeq$ by 
\begin{equation}
 d_l(a,b) = \sum_{W \in \walls^M(a|b)} l(W).
\end{equation}

\begin{prop} \label{prop:constructingMetricsOnMedianAlgebras}
 The map $d_l$ is a median metric on $M$ with intrinsic median operation $\mu_{d_l} = \mu$.
\end{prop}
\begin{sloppypar}
\begin{proof}
Clearly, $d_l$ is a metric (for the triangle inequality, observe that $\walls^M(x|z) \subseteq \walls^M(x|y) \cup \walls^M(y|z)$).
Moreover, if $x,y,z \in M$ with $z = \mu(x,y,z)$, then $\walls^M(x|y) = \walls^M(x|z) \sqcup \walls^M(z|y)$, so $d_l(x,y) = d_l(x,z) + d_l(y,z)$. Consequently, \cref{lem:medianCriterion} implies that $(M, d_l)$ is a median metric space with intrinsic median operation equal to $\mu$.
\end{proof}
\end{sloppypar}

Geometrically, one can picture this metric as being obtained from $\cube^{(1)}(M)$ by rescaling all edges which intersect a wall $W$ so that they have length $l(W)$ for all $W \in \walls(M)$.

\subsection{\texorpdfstring{From median metric spaces to $\cat(0)$ spaces}{From median metric spaces to CAT(0) spaces}} \label{subsec:fromMedianToCAT0}
Here we review a recent construction due to Bowditch~\cite[Section 7]{bowditch:somePropertiesOfMedianSpaces} which turns connected and complete median metric spaces into $\cat(0)$ spaces.
This allows us to apply the usual fixed point lemma for group actions with bounded orbits on $\cat(0)$ spaces in  the setting of these median metric spaces.

Suppose that $(X,d)$ is a complete and connected median metric space such that the median algebra $(X,\mu_d)$ has rank at most $n < \infty$. 
Let $Q \subseteq X$ be some $m$-cube, $m \leq n$ (cf.~\cref{subsec:medianAlgebras}).
Then $Q$ inherits a median metric from $(X,d)$ which comes from a function $l \colon \walls(Q) \to \Rgr$ as in \cref{subsec:constructingMetricsOnMedianAlgebras}.
Define $\omega(Q) = \sqrt{\sum_{W \in \walls(Q)} l(W)^2}$.

We say that $\{x,y\} \subseteq X$ is a \emph{diagonal} of a cube $Q \subseteq X$ if $x, y \in Q$ and $Q \subseteq \Intv_X^d(x,y)$.
A cube $Q$ is a \emph{maximal cube with diagonal $\{x,y\}$} if for all cubes $Q^\prime \supseteq Q$ for which $\{x,y\}$ is a diagonal, we have $Q = Q^\prime$.
For each $x,y \in X$, there is a \emph{unique} maximal cube $Q(x,y) \subseteq X$ (of dimension $\leq n$) with diagonal $\{x,y\}$, cf.~\cite[Section 5]{bowditch:somePropertiesOfMedianSpaces}.

Let $\mathbf{x} = (x_i)_{i=0}^N \subset X$ be some finite sequence in $X$ and define $\omega(\mathbf{x}) = \sum_{i=1}^N \omega(Q(x_{i-1},x_i))$.
Finally, for each $x,y \in X$, we set
\begin{equation*}
 \sigma_d(x,y) = \inf_{\mathbf{x}} \omega(\mathbf{x}),
\end{equation*}
where the infimum is taken over all finite sequences $\mathbf{x} = (x_i)_{i=0}^N$ with $x_0 = x$ and $x_N = y$.
The important fact~\cite[Theorem 1.1]{bowditch:somePropertiesOfMedianSpaces} is that $\sigma_d$ is a $\cat(0)$ metric on $X$ which satisfies $\frac{d(x,y)}{\sqrt{n}} \leq \sigma_d(x,y) \leq d(x,y)$ for all $x,y \in X$.

Note that $\sigma_d$ is defined canonically in terms of the median metric structure of $(X,d)$.
In particular, every isometry of $(X,d)$ is also an isometry of $(X, \sigma_d)$.

\begin{cor} \label{cor:fixedPointForMedianMetricSpace}
 Let a group act by isometries on a complete and connected median metric space of finite rank.
 Then the action has a fixed point if and only if it has bounded orbits.
\end{cor}
\begin{proof}
 In view of the previous discussion, this is a direct consequence of the fixed point lemma for actions with bounded orbits on complete $\cat(0)$ spaces~\cite[Part II, Corollary 2.8]{bridson_haefliger:nonPositiveCurvature}.
\end{proof}

\section{Metric median algebras} \label{sec:metricMedian}
\begin{defi}
\begin{definlist}
 \item  A \emph{metric median algebra} is a triple $(X,d, \median)$, where $(X, \median)$ is a median algebra and $(X,d)$ is a metric space such that $\median \colon X \times X \times X \to X$ is continuous with respect to the topology induced by the metric $d$.

  \item  If $(X,d,\median)$ is a metric median algebra, we write $\isom(X,d,\median)$ for the group of all isometries $f \colon (X,d) \to (X,d)$ which leave $\median$ invariant (that is, $f \circ \median = \median \circ (f \times f \times f)$).
  We call the elements of $\isom(X,d,\median)$ \emph{isometric automorphisms} of $(X,d,\mu)$.
\end{definlist}
\end{defi}
Every median metric space is a metric median algebra.
In fact, the intrinsic median operation of a median metric space is $1$-Lipschitz~\cite[Corollary 2.15]{chatterji_drutu_haglund:haagerupMedian}.

In \cref{subsec:rectifying}, we show that for a certain class of metric median algebras, including the ones considered in~\cite{bowditch:embeddingMedianAlgebrasIntoTrees}, one can construct a median metric which recovers the prescribed median algebra structure.
Moreover, the construction is canonical in the sense that the new metric is invariant with respect to all isometric automorphisms of the original metric median algebra.

In \cref{subsec:aFixedPointTheorem}, we apply these considerations together with the construction of Bowditch we have summarized in \cref{subsec:fromMedianToCAT0} in order to obtain a fixed point theorem for Kazhdan groups acting on metric median algebras.

\subsection{Rectifying metric median algebras} \label{subsec:rectifying}
Our following construction is inspired by the ideas developed in~\cite{bowditch:embeddingMedianAlgebrasIntoTrees}.

Let us fix a metric median algebra $(X,d,\mu)$.
We write $\finMedSub(X, \median)$ for the set of all finite median subalgebras of $(X,\median)$.
Since the median subalgebra generated by a finite subset is itself finite (cf.~\cref{lem:finiteSubalgebra}), $\finMedSub(X, \median)$ is a directed set with respect to the inclusion relation.

Let $M \in \finMedSub(X, \median)$.
Then for each edge $e$ in the $1$-skeleton of the $\cat(0)$ cubical complex $\cube(M)$, we define $\lambda^{M}(e) := d(e_-,e_+)$, where $e_\pm \in M$ are the boundary vertices of $e$.
Given a wall $W \in \walls(M)$, we consider its ``maximal thickness'',
 \begin{equation*}
  \lambda_{\max}^M(W) := \max_{e \pitchfork W} \lambda^M(e),
 \end{equation*}
 where the maximum is taken over all edges $e$ which intersect $W$ (recall that median walls correspond exactly to geometric hyperplanes in the cube complex).
 Notice that --- in contrast to~\cite{bowditch:embeddingMedianAlgebrasIntoTrees} --- we use the maximum here instead of the minimum. 
 This will be more convenient for our construction.
 
 For each $x,y \in X$ and $M \in \finMedSub(X, \mu)$, we define,
  \begin{equation}
  d_\median^M(x,y) = \begin{cases}
                  \sum_{W \in \walls^M(x|y)} \lambda^M_{\max}(W) & \text{if $x \in M$ and $y \in M$,} \\
                  0 & \text{otherwise.}
                 \end{cases} \label{eq:metricModFin}
 \end{equation}
 Note that we always have $d(x,y) \leq d^M_\median(x,y)$ as long as $x,y \in M$.
 We also observe that
  \begin{equation}
   d_\median^{g(M)}(g(x),g(y)) = d_\median^M(x,y), \label{eq:finEqui}
  \end{equation}
    for all $M \in \finMedSub(X,\mu)$, $x,y \in X$ and $g \in \isom(X,d,\median)$.
    However, for this to be true it is crucial that $g$ preserves both the metric and the median structure.

For each $M \in \finMedSub(X, \mu)$, the restriction of $d_\mu^M$ to $M \times M$ is a distance function of the type considered \cref{subsec:constructingMetricsOnMedianAlgebras}. 
 Thus, \cref{prop:constructingMetricsOnMedianAlgebras} implies that it is a median metric on $M$ with intrinsic median operation equal to the appropriate restriction of $\mu$.
 
 \begin{defi} \label{defi:rectifiable}
   Let $(X,d, \median)$ be a metric median algebra. For $x,y \in X$, we define
   \begin{equation}
    d_\median(x,y) = \sup \left\{ d_\median^M(x,y) \middlemid M \in \finMedSub(X,\mu) \right\} \in [0, \infty]. \label{eq:metricMod}
   \end{equation}
   We call a metric median algebra $(X,d,\median)$ \emph{rectifiable} if $d_\median(x,y) < \infty$ for all $x,y \in X$.
   
   Moreover, we say $(X,d,\median)$ is \emph{uniformly rectifiable} if there exists a constant $L \in \Rgeq$ such that $d_\median(x,y) \leq L d(x,y)$ for all $x,y \in X$.
 \end{defi} 
 
 \Cref{eq:finEqui} implies that $d_\mu(g(x),g(y)) = d_\mu(x,y)$ for all $x,y \in X$ and $g \in \isom(X,d,\mu)$.
 In other words, $d_\mu$ is $\isom(X,d,\mu)$-invariant.
 
 We also note that if $(X,d)$ is a median metric space with intrinsic median operation $\mu$, then $d_\mu = d$.

 The main result of this subsection is the following.
 \begin{prop} \label{prop:deformMetric}
 Let $(X,d, \median)$ be a rectifiable metric median algebra. Then $d_\median$ as in \labelcref{eq:metricMod} defines an $\isom(X,d,\mu)$-invariant median metric on $X$. 
 Its intrinsic median operation coincides with $\median$.
 Moreover, if $(X,d,\median)$ is uniformly rectifiable, then $d_\median$ is bi-Lipschitz equivalent to the original metric $d$.
\end{prop}
 The key to the proof of the proposition is the following lemma.

\begin{lem} \label{lem:monotonic}
   Let $(X,d,\median)$ be a metric median algebra and $M, N \in \finMedSub(X,\median)$ satisfying $M \subseteq N$. 
   Then we have $d_\median^M(x,y) \leq d_\median^N(x,y)$ for all $x,y \in X$.
   In other words, the net $M \mapsto d_\median^M(x,y)$ (defined on the directed set $\finMedSub(X,\median)$) is monotonically non-decreasing for fixed $x,y \in X$.
  \end{lem}
  \begin{proof}
  We may assume that both $x$ and $y$ are contained in $M$, otherwise $d_\median^M(x,y) = 0$ and the assertion is clear.
  
  Consider the set of walls  $\walls(M \subset N) := \bigcup_{a,b \in M} \walls^N(a|b) \subseteq \walls(N)$.
  In other words, $\walls(M \subset N)$ is the set of walls in $N$ which induce non-trivial walls on $M$.
  More precisely, there is a map $\iota^\walls \colon \walls(M \subset N) \to \walls(M)$, which in terms of half spaces is defined by $\iota^\walls(\{h,h^c\}) = \{h \cap M, h^c \cap M\}$.
  Observe that for every $a,b \in M$, 
  \begin{equation}
   \walls^N(a|b) = (\iota^{\walls})^{-1}(\walls^M(a|b)). \label{eq:wallPreimage}
  \end{equation}

  We claim that for each $W \in \walls(M)$, we have,
  \begin{equation}
   \lambda_{\max}^M(W) \leq \sum_{V \in (\iota^{\walls})^{-1}(W)} \lambda_{\max}^N(V). \label{eq:monotonicForWalls}
  \end{equation}
  Indeed, choose elements $a,b \in M$ which are adjacent in the cube complex $\cube(M)$, are separated by $W$, and attain the maximum $d(a,b) = \lambda_{\max}^M(W)$. Let $a = a_0, a_1, \dotsc, a_N = b \subseteq N = \cube^{(0)}(N)$ be a combinatorial geodesic connecting $a$ to $b$ in $\cube(N)$.
  By the triangle inequality, we conclude,
  \begin{equation*}
      \lambda_{\max}^M(W) = d(a,b) \leq \sum_{i=1}^N d(a_{i-1},a_i) \leq \sum_{V \in (\iota^{\walls})^{-1}(W)} \lambda_{\max}^N(V),
  \end{equation*}
  where the second inequality follows from \labelcref{eq:wallPreimage} (using $\walls^M(a|b) = \{W\}$) and the definition of $\lambda_{\max}^N$.
  
  Using \labelcref{eq:monotonicForWalls,eq:wallPreimage}, we finish the proof of \cref{lem:monotonic},
  \begin{align*}
   d_\median^M(x,y) = \sum_{W \in \walls^M(x|y)} \lambda^M_{\max}(W) &\leq \sum_{W \in \walls^M(x|y)} \sum_{V \in (\iota^{\walls})^{-1}(W)} \lambda_{\max}^N(V) \\
   &= \sum_{V \in (\iota^{\walls})^{-1}(\walls^M(x|y))} \lambda_{\max}^N(V) \\
   &= d_\median^N(x,y). \qedhere
  \end{align*}
  \end{proof}

\begin{proof}[Proof of \cref{prop:deformMetric}]
Abbreviate $\finMedSub := \finMedSub(X,\median)$.
By \cref{lem:monotonic}, the net of real numbers $\left( d_\median^M(x,y) \right)_{M \in \finMedSub}$ is monotonically non-decreasing for fixed $x,y$. 
By assumption, it is also bounded above, whence its limit exists and
  \begin{equation}
   d_\median(x,y) = \lim_{\mathscr{M}} d_\median^M(x,y). \label{eq:metricAsLimit}
  \end{equation}
 It is clear that $d_\mu$ is a metric and the $\isom(X,d,\mu)$-invariance has already been noted in the discussion before the proposition.
 
 To see that $d_\mu$ is a median metric, let $x,y,z \in X$ with $z = \mu(x,y,z)$.
 Since $d_\median^M$ is a median metric on $M$ with median operation $\mu$, we have
 $d_\median^M(x,z) + d_\median^M(z,y) = d_\median^M(x,y)$ for all $M \in \finMedSub$ with $x,y,z \in M$.
 This equality passes to the limit in \labelcref{eq:metricAsLimit}, whence the hypothesis of \cref{lem:medianCriterion} is satisfied for $\rho = d_\median$.
 We conclude that $(X,d_\mu)$ is a median metric space with intrinsic median operation equal to $\mu$.
 
  If $(X,d,\mu)$ is uniformly rectifiable with constant $L$, then $d(x,y) \leq d_\mu(x,y) \leq L d(x,y)$ for all $x,y \in X$, which proves the last statement.
\end{proof}

We finish this section with a technical lemma which directly follows from Bowditch's work and will be required in \cref{sec:homosFromKazhdan}.

\begin{lem}[{cf.~\cite[Lemma 6.2]{bowditch:embeddingMedianAlgebrasIntoTrees}}] \label{lem:conditionsForBeingRecitfiable}
 Let $(X,d,\mu)$ be a metric median algebra of finite rank $n$.
 Suppose that 
 
 \begin{enumerate}[(i)]
           \item there exists a constant $k > 0$ such that \label{item:Lipschitz}
   \begin{gather*}
  d(\median(a,b,c), \median(a,b,\tilde{c})) \leq k d(c, \tilde{c}) \quad \text{for all $a,b,c,\tilde{c} \in X$,} 
  \end{gather*}
  
           \item there exists a constant $l \geq 1$ such that for all $x,y \in X$, there is an $l$-Lipschitz map $c \colon [0, d(x,y)] \to X$ with $c(0) = x$ and $c(d(x,y)) = y$. \label{item:almostGeodesic}
  \end{enumerate}

  Then $(X,d,\mu)$ is uniformly rectifiable with constant $L$ depending only on $n$,$k$ and $l$.
\end{lem}
\begin{proof}
 To see how the claim follows from Bowditch' work, we introduce the following temporary notation:
 Let $d_{\mu,\min}^M$ be defined in the same way as $d_\mu^M$, only we replace $\lambda_{\max}^M$ in \labelcref{eq:metricModFin} by
  \begin{equation*}
  \lambda_{\min}^M(W) := \min_{e \pitchfork W} \lambda^M(e), \quad M \in \finMedSub(X,\mu), W \in \walls(M).
 \end{equation*}
 Then $d_{\mu,\min}^M$ agrees with \enquote{$\lambda$} in the notation of~\cite{bowditch:embeddingMedianAlgebrasIntoTrees}.
 Now~\cite[Lemma 6.2]{bowditch:embeddingMedianAlgebrasIntoTrees} shows\footnote{The tacit assumption of colourability in Bowditch's work is not used in this lemma, since one could restrict to working with intervals in this case.} that there exists a constant $\widetilde{L} = \widetilde{L}(n,k,l)$ such that
 \begin{equation*}
  d_{\mu,\min}^M(x,y) \leq \widetilde{L} d(x,y)
 \end{equation*}
 for all $x,y \in X$.
 Moreover, assumption~\cref{item:Lipschitz} implies that $\lambda^M(e) \leq k \lambda^M(\tilde{e})$ for all pairwise parallel edges $e, \tilde{e}$ in the cube complex corresponding to $M$.
 In particular, $\lambda_{\max}^M(W) \leq k \lambda_{\min}^M(W)$ for all $M \in \finMedSub(X,\mu), W \in \walls(M)$.
 Consequently, $d_{\mu}^M(x,y) \leq k \widetilde{L} d(x,y)$ for all $x,y \in X$, and the same holds for $d_\mu$.
\end{proof}

\subsection{A fixed point theorem} \label{subsec:aFixedPointTheorem}

Applying the previous construction and \cref{subsec:fromMedianToCAT0}, we obtain the following fixed point results.

\begin{prop} \label{prop:fixedPoint}
 Let $(X,d,\mu)$ be a uniformly rectifiable metric median algebra of finite rank and suppose that $(X,d)$ is complete and connected.
 Let $G$ be a group acting by isometric automorphisms on $(X,d,\mu)$.
 If the action has bounded orbits, then it has a fixed point.
\end{prop} 
\begin{proof}
 By \cref{prop:deformMetric}, $G$ acts by isometries on $(X, d_\mu)$. Since $d$ and $d_\mu$ are bi-Lipschitz equivalent, the action has bounded orbits with respect to $d_\mu$ if it has with respect to $d$.
 So the proposition follows from \cref{cor:fixedPointForMedianMetricSpace}.
\end{proof}

\begin{thm} \label{thm:TfixedPoint}
 Every action of a group with Kazhdan's property (T) by isometric automorphisms on a complete, connected, uniformly rectifiable metric median algebra of finite rank has a fixed point.
\end{thm}
\begin{proof}
Kazhdan's property (T) implies that any isometric action on a median metric space must have bounded orbits \cite[Theorem 1.3]{nica:groupActionsOnMedianSpaces}~\cite[Theorem 1.2]{chatterji_drutu_haglund:haagerupMedian}. 
Hence, the claim follows by the same argument as in the proof of \cref{prop:fixedPoint}.
\end{proof}

\section{Coarse median structures and quasi-isometries} \label{sec:qiInvariance}
Let $S$ be a set and $(X,d)$ a metric space.
Let $f, f^\prime \colon S \to X$ be some maps and $K \geq 0$. 
We say $f$ and $f^\prime$  are \emph{$K$-close}, and write $f \coarseEq_K f^\prime$, if $d(f(s), f^\prime(s)) \leq K$ for all $s \in S$.
Moreover, two such maps $f, f^\prime$ are said to be \emph{close}, denoted by $f \coarseEq f^\prime$, if they are $K$-close for some $K \geq 0$.
Closeness is an equivalence relation on $X^S$.

\begin{defi}
 Let $(X,d)$ be a metric space, $S$ a set, and $\mu_Y \colon Y^3 \to Y$ some ternary operation for each $Y \in \{ X, S \}$.
 A map $\phi \colon S \to X$ is called a \emph{$K$-quasi-morphism} if $  \phi \circ \mu_S \coarseEq_K \mu_X \circ (\phi \times \phi \times \phi)$.
\end{defi}

In terms of this notation, the definition of a coarse median (cf.~\cref{subsec:introToCoarseMedian}) reads as follows.
\begin{defi}[Coarse median~\cite{bowditch:coarseMedian}] \label{defi:coarseMedian}
 Let $(X, d)$ be a metric space. A ternary operation $\mu \colon X^3 \to X$ is called a \emph{coarse median} if there is a constant $k \in \Rgeq$ and a non-decreasing function $h \colon \N \to \Rgeq$ such that the following conditions are satisfied.
 
 \begin{definlist}
  \item For every $x, y, z, x^\prime, y^\prime, z^\prime \in X$, we have,
  \begin{equation*} 
   d\left( \mu(x, y, z), \mu(x^\prime, y^\prime, z^\prime) \right) \leq k \left( d(x,x^\prime) + d(y, y^\prime) + d(z, z^\prime) \right) + h(0).
  \end{equation*}
  
  \item For every finite non-empty subset $A \subseteq X$, there is a finite median algebra $(M, \mu_M)$ together with a $h(|A|)$-quasi-morphism $\lambda \colon M \to X$ and a map $\pi \colon A \to M$ such that 
  \begin{equation*}
   \iota_A \sim_{h(|A|)} \lambda \circ \pi,
  \end{equation*}
  where $\iota_A$ denotes the inclusion map $A \hookrightarrow X$.
 \end{definlist}
 We call $k$ and $h$ \emph{parameters} of the coarse median $\mu$.
 
 Furthermore, if it is possible to fix parameters such that the cube complexes $\cube(M)$ have dimension at most $n$ for all median algebras $M$ appearing in the definition, then we say that the coarse median $\mu$ has \emph{rank at most $n$}.
 If it has rank at most $n$ but not rank at most $n-1$, then we say it has \emph{rank $n$}.
\end{defi}

\begin{defi} \label{defi:cmStructures}
 A \emph{coarse median structure} on a metric space $X$ is a closeness class of coarse medians $\mu \colon X^3 \to X$.
 We write $[\mu]$ for the coarse median structure represented by a coarse median $\mu$.
\end{defi}

The concept of rank is well-defined on the level of coarse median structures.

\begin{rem} \label{rem:WLOGIfDontCareAboutRank}
 If we do not care about the rank, we may always assume (after changing the parameters) that $\iota_A = \lambda \circ \pi$ in \cref{defi:coarseMedian} \labelcref{item:cM2}.
 This has already been suggested in~\cite[Section 8]{bowditch:coarseMedian}; here we work out the claim explicitly and show that it suffices to allow for increasing the rank by one.
 
 Indeed, let $A \subseteq X$ be a finite subset of a coarse median space, and $\lambda \colon M \to X$, $\pi \colon A \to M$ as in \cref{defi:coarseMedian} \labelcref{item:cM2}. 
 Let $T$ be some finite tree with vertex set $T^{(0)} = A$, and $\mu_T \colon A^3 \to A$ the corresponding rank $1$ median structure.
 Define
 \begin{equation*}
  \tilde{\lambda} \colon M \times A \to X, \quad \tilde{\lambda}(y,a) = \begin{cases}
                                                                               a & \text{if $y = \pi(a)$,} \\
                                                                               \lambda(y) & \text{otherwise.}
                                                                              \end{cases}
 \end{equation*}
  Let $\proj_M \colon M \times A \to M$ be the projection on the first factor, and observe that by $\iota_A \sim_{h(|A|)} \lambda \circ \pi$, we have, 
  \begin{equation}
   \tilde{\lambda} \sim_{h(|A|)} \lambda \circ \proj_M. \label{eq:newLambdaEstimate}
  \end{equation}
  Moreover, define $\tilde{\pi} \colon A \to M \times A$ by $\tilde{\pi}(a) = (\pi(a),a)$ for all $a \in A$.  Clearly, we have $(\tilde{\lambda} \circ \tilde{\pi})(a) = a$ for all $a \in A$.
  Since $\lambda$ is an $h(|A|)$-quasi-morphism, it follows that
  \begin{equation*}
   \tilde{\lambda} \circ \mu_{M \times A} \sim_{h(|A|)} \lambda \circ \proj_{M} \circ \mu_{M \times A} = \lambda \circ \mu_{M} \circ \proj_M^{\times 3} \sim_{h(|A|)} \mu \circ \lambda^{\times 3} \circ \proj_M^{\times 3} =  \mu \circ (\lambda \circ \proj_M)^{\times 3},
  \end{equation*}
  where we have used the shorthand $f^{\times 3} := f \times f \times f$.
  By \labelcref{eq:newLambdaEstimate} and \cref{defi:coarseMedian} \labelcref{item:cM1}, we have $\mu \circ (\lambda \circ \proj_M)^{\times 3} \sim_{k h(|A|) + h(0)} \mu \circ \tilde{\lambda}$. Thus, $\tilde{\lambda}$ is an $((k+2) h(|A|) + h(0))$-quasi-morphism.
  This proves the claim, as we can replace $\lambda, \pi$ by $\tilde{\lambda}, \tilde{\pi}$ if we also replace $h$ by $n \mapsto (k+2)h(n) + h(0)$. 
  However, this construction has increased the rank of the median algebra on which our \enquote{new $\lambda$} is defined by one.
\end{rem}

\subsection{Quasi-isometry invariance}
The existence of a coarse median structure (of certain rank) is a quasi-isometry invariant of the underlying metric space~\cite[Lemma 8.1]{bowditch:coarseMedian}.
We formalize this observation by introducing the notions of ``pushforward'' and ``pullback'' of a coarse median structure via a quasi-isometry.

 \begin{defi} \label{defi:inducedMedian}
  Let $f \colon X \to Y$ be a quasi-isometry between two metric spaces and $[\mu_X]$, $[\mu_Y]$ coarse median structures on $X$, respectively on $Y$. We choose a quasi-inverse $g \colon Y \to X$ to $f$, and define the \emph{pushforward} of $[\mu_X]$ via $f$ by
  \begin{align*}
   f_* [\mu_X] :=& \left[f \circ \mu_X \circ (g \times g \times g) \right], \\
   \intertext{and similarly, the \emph{pullback} of $[\mu_Y]$ via $f$ by}
   f^* [\mu_Y] :=& \left[g \circ \mu_Y \circ (f \times f \times f) \right].
  \end{align*}
 \end{defi}
 One readily verifies that $f \circ \mu_X \circ (g \times g \times g)$ (respectively $g \circ \mu_Y \circ (f \times f \times f)$) are coarse medians on $Y$ (respectively $X$). Moreover, the coarse median structures represented by these expressions do not depend on the choice of quasi-inverse $g$ or the choice of representatives for $[\mu_X]$ and $[\mu_Y]$. Hence $f_* [\mu_X]$ and $f^* [\mu_Y]$ are well-defined coarse median structures.

 By straightforward estimates, we obtain the following formal properties.
\begin{prop} \label{prop:formalPropOfInducedCM}
  Under the notation of \cref{defi:inducedMedian}, we have:
  
  \begin{thminlist} 
   
   \item \label{item:inducedCMrank}$\rk f_* [\mu_X]= \rk [\mu_X]$ and $\rk f^* [\mu_Y] = \rk [\mu_Y]$.
   
   \item If $g$ is some quasi-inverse to $f$, then $f_* [\mu_X] = g^* [\mu_X]$ and $f^* [\mu_Y] = g_* [\mu_Y]$. \label{item:inducedCMSymmetric}
   
   \item \label{item:inducedCMfunctor} If $f_1 \colon X \to Z$ and $f_2 : Z \to Y$ are quasi-isometries, then $ (f_2 \circ f_1)_* [\mu_X] = {f_2}_* {f_1}_* [\mu_X]$ and $(f_2 \circ f_1)^* [\mu_Y] = f_1^* f_2^* [\mu_Y]$. 
   
   \item \label{item:inducedCMequFunc}If $f \sim f^\prime$, then $f_* [\mu_X] = f^\prime_* [\mu_X]$ and $f^* [\mu_Y] = (f^\prime)^* [\mu_Y]$.

  \end{thminlist}
\end{prop}

 Given a metric space $(X,d)$, we denote the set of quasi-isometries $X \to X$ modulo the equivalence relation $\sim$ by $\qisom(X,d)$.
 Composition of maps induces a well-defined binary operation on $\qisom(X,d)$ which satisfies the group axioms.
 The resulting group is known as the \emph{quasi-isometry group} of $(X,d)$.

 By the proposition above, it follows that for a coarse median space $(X,d)$, the group $\qisom(X,d)$ acts on the set of coarse median structures on $X$.
 Indeed, if $[f] \in \qisom(X,d)$ is represented by a quasi-isometry $f \colon X \to X$, and $[\mu]$ a coarse median structure on $X$ , we define the action of $[f]$ on $[\mu]$ by $[f] \cdot [\mu] := f_* [\mu]$.

 In view of the natural homomorphism $\isom(X,d) \to \qisom(X,d)$, the isometry group of a metric space also acts on the set of coarse median structures.
 The next example demonstrates that even this action is not necessarily trivial.

\begin{ex} \label{ex:coarseMedianOnEucledianSpace}
Consider the intrinsic median operation $\mu_{\R^n}$ on $\R^n$ associated to the metric induced by the $1$-norm, $\|(x^i)\|_1 = \sum_{i=1}^n |x^i|$.
Let $\euclSpace^n$ denote the metric space determined by $\R^n$ endowed with the Euclidean norm $\|(x^i)\|_2 = \sqrt{\sum_{i=1}^n |x^i|^2}$.
Then $\mu_{\R^n}$ is a coarse median on $\euclSpace^n$ as well (since all norms on $\R^n$ are equivalent).
 
 Let $A \in \isom(\euclSpace^2)$ be the rotation by $\frac{\pi}{4}$ around the origin $(0,0)$.
 For $k \in \N$, consider $x_k = (k,0), y_k = (0,k) \in \euclSpace^2$. 
 Then, by elementary Euclidean geometry, we obtain $A x_k  = ( \frac{k}{\sqrt{2}}, \frac{k}{\sqrt{2}})$ and $A y_k = (- \frac{k}{\sqrt{2}}, \frac{k}{\sqrt{2}})$. 
 Thus
 \begin{align*}
  \mu_{\R^2}(x_k, y_k, 0) &= (0, 0), \\
  \mu_{\R^2}(A x_k, A y_k, 0) &= (0, \tfrac{k}{\sqrt{2}}),
 \end{align*}
  and subsequently $\| \mu_{\R^2}(A x_k, A y_k, A0) - A (\mu_{\R^2}(x_k, y_k, 0)) \|_2 \to \infty$ for $k \to \infty$, which means $A^* [\mu_{\R^2}] \neq [\mu_{\R^2}]$.
  In fact, if  $A_\varphi$ denotes the rotation by $\varphi$ around the origin, then a similar computation as above shows that $A_\varphi^* [\mu_{\R^2}] \neq [\mu_{\R^2}]$ provided that $\varphi$ is not an integer multiple of $\frac{\pi}{2}$.
  In particular, there are uncountably many different but isomorphic coarse median structures on $\euclSpace^n$, $n \geq 2$.
\end{ex}

This example shows that, in general, a coarse median structure is an additional structure imposed on a metric space (and does not necessarily arise from the metric structure itself as in the case of hyperbolic spaces, cf.~\cref{subsec:hyperbolicSpaces}).

\begin{defi}
 Let a group $G$ act by isometries on a metric space $X$.
 A coarse median structure $[\mu]$ on $X$ is called \emph{uniformly $G$-invariant} if 
 \begin{equation*}
  \sup_{x_i \in X, g \in G} d(\mu(g \cdot x_1, g \cdot x_2, g \cdot x_3), g \cdot \mu(x_1,x_2,x_3)) < \infty.
 \end{equation*}
\end{defi}

\begin{defi} \label{defi:uniformLeftInvariant}
 Let $G$ be a group and $S \subseteq G$ a finite generating set.
 Denote by $d_S$ the left-invariant word metric on $G$ associated to $S$.
 A coarse median structure $[\mu]$ on $(G,d_S)$ is called \emph{uniformly left-invariant} if it is uniformly $G$-invariant with respect to the left multiplication action of $G$ on itself.
\end{defi}

\subsection{\texorpdfstring{$\delta$-hyperbolic spaces}{delta-hyperbolic spaces}} \label{subsec:hyperbolicSpaces}
\begin{sloppypar}Every $\delta$-hyperbolic geodesic space admits a natural rank $1$ coarse median structure~\cite[Section 3]{bowditch:coarseMedian}.
For us, a geodesic space is \emph{$\delta$-hyperbolic} (for some fixed $\delta \geq 0$) if all its geodesic triangles are $\delta$-slim~\cite[Chapter III.H]{bridson_haefliger:nonPositiveCurvature}.
In the following, we review Bowditch's construction of a coarse median on such a space.
\end{sloppypar}
Let $(X,d)$ a geodesic metric space and $K \geq 0$.
A \emph{$K$-center} of a geodesic triangle in $X$ is a point which lies within distance $K$ of each side of the triangle.
Requiring the existence of $K$-centers in geodesic metric spaces is an equivalent formulation of $\delta$-hyperbolicity for geodesic metric spaces~\cite{bowditch:hyperbolicity,bridson_haefliger:nonPositiveCurvature}.
However, we only need the facts which are collected in the lemma below.
\begin{lem} \label{lem:centersInHyperbolicSpaces}
  Let $(X,d)$ be a $\delta$-hyperbolic geodesic metric space.
  Then the following holds.
  \begin{thminlist}
  
   \item Every geodesic triangle in $X$ admits a $\delta$-center. \label{item:KcenterExists}
   
   \item For every constant $K \in \Rgeq$, there exists a constant $K^\prime = K^\prime(K, \delta)$ such that: If $m_1$ and $m_2$ are $K$-centers of geodesic triangles $\bigtriangleup_1$ and $\bigtriangleup_2$, respectively, where $\bigtriangleup_1$ and $\bigtriangleup_2$ have the same vertices, then $d(m_1, m_2) \leq K^\prime$. \label{item:KcenterUnique}
  \end{thminlist}
\end{lem}

 \begin{sloppypar}Let $(X,d)$ be a $\delta$-hyperbolic geodesic metric space and fix some $K \geq \delta$.
   For each $x_1,x_2,x_3 \in X$, choose an arbitrary $K$-center of some geodesic triangle with vertices $x_1,x_2,x_3$ and denote it by $\mu_{\hypCM}(x_1,x_2,x_3)$.
  By \cref{lem:centersInHyperbolicSpaces}, the closeness class of $\mu_{\hypCM} \colon X^3 \to X$ does not depend on any of the choices (not even on $K$ as long as it remains fixed throughout the construction).
   Moreover, using the fact that in hyperbolic spaces finite subsets can be approximated by trees~~\cite[\S 6.2]{gromov:hyperbolicGroups}, one proves that $\mu_{\hypCM}$ is a coarse median of rank at most $1$~\cite[Section 3]{bowditch:coarseMedian}.
\end{sloppypar}   
   \begin{defi}
     We call $[\mu_\hypCM]$ the \emph{hyperbolic coarse median structure} of $X$.
   \end{defi}

 Using stability of quasi-geodesics~\cite[Part III, Theorem 1.7]{bridson_haefliger:nonPositiveCurvature}, we obtain the following lemma which implies that the hyperbolic coarse median structure is \enquote{stable under quasi-isometries}.
 This is in strong contrast to what we have seen in \cref{ex:coarseMedianOnEucledianSpace}.

\begin{lem} \label{lem:hypCMNatural}
 Let $f \colon X \to Y$ be a quasi-isometry between $\delta$-hyperbolic geodesic metric spaces. Let $[\mu_{\hypCM,X}], [\mu_{\hypCM,Y}]$ denote the hyperbolic coarse median structures on $X$ and $Y$, respectively. Then $f_* [\mu_{\hypCM,X}] = [\mu_{\hypCM,Y}]$.
\end{lem}

For any hyperbolic group $G$ and any finite generating set $S \subseteq G$, there is a well-defined hyperbolic coarse median structure $\boldsymbol{\mu}_G$  on $(G, d_S)$.
Indeed, $\boldsymbol{\mu}_G$ is obtained as the pullback of the hyperbolic coarse median structure on the Cayley graph of $G$ via the inclusion map of the vertex set.
From \cref{lem:centersInHyperbolicSpaces}, we obtain the following.
\begin{cor}
 The coarse median structure $\boldsymbol{\mu}_G$ on any hyperbolic group $G$ is uniformly left-invariant.
\end{cor}

We conclude this section by illustrating our formal tools and use them to compare coarse median structures on hyperbolic groups acting properly and cocompactly on a $\cat(0)$ cubical complex.
For instance, such examples arise from various types of small-cancellation groups~\cite{wise:cubulatingSmallCancel}.

\begin{ex} \label{ex:HyperbolicGpactingOnCCC}

Fix a finite-dimensional and locally finite $\cat(0)$ cubical complex $C$. 
The vertex set $C^{(0)}$, endowed with the path metric from the $1$-skeleton of $C$, is a median metric space. The inclusion map $\iota \colon C^{(0)} \hookrightarrow C$ is a quasi-isometry. 
Thus, there is a coarse median structure on $C$ defined by $\boldsymbol{\mu}_C := \iota_* [\mu_{C^{(0)}}]$, where $\mu_{C^{(0)}}$ is the median on the vertex set.
Furthermore, let $G$ be a hyperbolic group which acts properly and cocompactly by isometries on $C$. 
By the \v{S}varc--Milnor lemma~\cite[Proposition I.8.19]{bridson_haefliger:nonPositiveCurvature}, the map $\ev_{x} \colon G \to C, g \mapsto g \cdot x$ is a quasi-isometry for all $x \in C$.

 We claim that for any choice of base-point $x \in C$, we have,
 \begin{equation}
 (\ev_{x})_* \boldsymbol{\mu}_G = \boldsymbol{\mu}_C. \label{eq:hyperbolicCoarseMedianIsCubicalCoarseMedian}
\end{equation}  
 In other words, the hyperbolic coarse median structure on $G$ automatically agrees with the one induced by the median on the $\cat(0)$ cubical complex.
 
 Indeed, first observe that $C$ is a hyperbolic space due to quasi-isometry invariance of hyperbolicity~\cite[Chapter III.H, Theorem 1.9]{bridson_haefliger:nonPositiveCurvature}.
 We consider the $1$-skeleton $C^{(1)}$ to be endowed with the path metric, which we denote by $d_{C^{(1)}}$. Then the inclusion $j \colon C^{(1)} \hookrightarrow C$ is a quasi-isometry.	
 Thus $(C^{(1)}, d_{C^{(1)}})$ is a hyperbolic geodesic metric space as well.
 Let $\boldsymbol{\mu}_{C^{(1)}, \hypCM}$ denote the hyperbolic coarse median structure on $C^{(1)}$. By \cref{lem:hypCMNatural}, we have
 \begin{equation}
    j_* \boldsymbol{\mu}_{C^{(1)}, \hypCM} = (\ev_{x})_* \boldsymbol{\mu}_G. \label{eq:CMcomparison1}
 \end{equation}
  Furthermore, we have 
  \begin{equation}
  i_* [\mu_{C^{(0)}}] = \boldsymbol{\mu}_{C^{(1)}, \hypCM}, \label{eq:CMcomparison2}
  \end{equation}
 where $i \colon C^{(0)} \hookrightarrow C^{(1)}$ is the inclusion map.
 To see this, just observe that for every $x_1,x_2,x_3 \in C^{(0)}$, there exists a geodesic tripod in $C^{(1)}$ with vertices $x_1,x_2,x_3$ and center $\mu_{C^{(0)}}(x_1,x_2,x_3)$.
 As a consequence, $\mu_{C^{(0)}}(x_1,x_2,x_3)$ is a $0$-center of this particular geodesic triangle and hence valid choice for an hyperbolic coarse median point (cf.\ the discussion in \cref{subsec:hyperbolicSpaces}) of $(x_1,x_2,x_3)$ in $C^{(1)}$.
 This proves $\labelcref{eq:CMcomparison2}$.
  Using \labelcref{eq:CMcomparison1}, \labelcref{eq:CMcomparison2} and \cref{prop:formalPropOfInducedCM}, we conclude \labelcref{eq:hyperbolicCoarseMedianIsCubicalCoarseMedian} as follows,
 \begin{equation*}
  \boldsymbol{\mu}_C = \iota_* [\mu_{C^{(0)}}] = j_* i_* [\mu_{C^{(0)}}] = j_* \boldsymbol{\mu}_{C^{(1)}, \hypCM} = (\ev_{x})_* \boldsymbol{\mu}_G.
 \end{equation*}
\end{ex}

\section{Homomorphisms from Kazhdan groups into coarse median groups} \label{sec:homosFromKazhdan}
In this section, we prove our main result on homomorphisms from Kazhdan groups into groups with an uniformly left-invariant coarse median structure.
In general, our proof works the same way as in Paulin's article~\cite{paulin:outerAutomorphismsOfHyperbolicGroups} on the outer automorphism group of hyperbolic groups with Kazhdan's property (T).
However, in a slight variation compared to Paulin, we use the language of asymptotic cones instead of equivariant Gromov--Hausdorff topologies.
\subsection{Asymptotic cones} \label{subsec:asympCone}
We start by reviewing basic facts about ultralimits and asymptotic cones~\cite[Chapter I.5]{bridson_haefliger:nonPositiveCurvature}.
Let $\omega$ be a non-principal ultrafilter on $\N$ and $(X_n, d_n, \basePt_n)_{n \in \N}$ a sequence of pointed metric spaces.
Its \emph{ultralimit} with respect to $\omega$ is defined by
\begin{equation*}
 \mathrm{X}_\omega = \left\{ (x_n)_{n \in \N} \in \prod_{n \in \N} X_n \middlemid \sup_{n \in \N} d_n(x_n, \basePt_n) < \infty \right\} \bigg/ \mathord\sim,
\end{equation*}
where $(x_n)_n \sim (y_n)_n$ if $d_\omega((x_n)_n, (y_n)_n) := \lim_\omega d_n(x_n,y_n) = 0$. 
Then $d_\omega$ descends to a complete metric on $\mathrm{X}_\omega$. 

\begin{sloppypar}For every $n \in \N$, let $\alpha_n \colon X_n \to X_n$ be an isometry and assume that $\sup_{n \in \N} d_n( \basePt_n, \alpha_n(\basePt_n) ) < \infty$.
Then there is an isometry $\alpha_\omega \colon \mathrm{X}_\omega \to \mathrm{X}_\omega$ defined on representatives by $\alpha_\omega((x_n)_n) = (\alpha_n(x_n))_n$.
\end{sloppypar}
If we apply this construction to a sequence of the form $(X, \frac{1}{\lambda_n} d, \basePt_n)$, where $(X, d)$ is some fixed metric space with a sequence $(\basePt_n)_n \subset X$, and $(\lambda_n)_n \subset \Rgeq$ some sequence with $\lambda_n \to \infty$ as $n \to \infty$, then we call the resulting space $\mathrm{X}_\omega$ an \emph{asymptotic cone} of $(X,d)$, denoted by $\assympCone_\omega(X,d, (\lambda_n)_n, (\basePt_n)_n)$.

\begin{prop} \label{prop:asymptoticConeIsMetricMedianAlgebra}
 Let $(X,d)$ be a geodesic metric space which admits a coarse median structure of rank at most $n$.
 Then every asymptotic cone of $X$ is a uniformly rectifiable metric median algebra of rank at most $n$.
\end{prop}
\begin{proof}
 Fix $\lambda_n \to \infty$ and $(\basePt_n)_n \subset X$ and set $\mathrm{X}_\omega = \assympCone_\omega(X, d, (\lambda_n)_n, (\basePt_n)_n)$.
  Let $\mu$ be a coarse median representing the given coarse median structure on $X$.

 Now consider elements $\mathrm{x}, \mathrm{y}, \mathrm{z} \in \mathrm{X}_\omega$, represented by sequences $(x_n)_n, (y_n)_n,(z_n)_n \subset X$, respectively.
 Then $(\mu(x_n, y_n, z_n))_n$ represents an element of $\mathrm{X}_\omega$, which we denote by ${\mu_\omega}(\mathrm{x}, \mathrm{y}, \mathrm{z})$.
 Then $(\mathrm{X}_\omega, d_\omega, \mu_\omega)$ is a well-defined metric median algebra of rank at most $n$~\cite[Proposition 9.4]{bowditch:coarseMedian}.
 Moreover, $\mu_\omega$ is Lipschitz~\cite[Lemma 9.2]{bowditch:coarseMedian}, that is, it satisfies condition \labelcref{item:Lipschitz} from \cref{lem:conditionsForBeingRecitfiable}.
 As an asymptotic cone of a geodesic space, $(\mathrm{X}_\omega, d_\omega)$ is geodesic itself, so it trivially satisfies \labelcref{item:almostGeodesic} from \cref{lem:conditionsForBeingRecitfiable}.
 Thus, \cref{lem:conditionsForBeingRecitfiable} implies that $(\mathrm{X}_\omega, d_\omega, \mu_\omega)$ is uniformly rectifiable.
\end{proof}

The following technical lemma follows immediately from the definitions, we record it here for later reference.
\begin{lem} \label{lem:InducingIsometricAutomorphisms}
 We keep working with the setup of \cref{prop:asymptoticConeIsMetricMedianAlgebra}.
 Consider an asymptotic cone $\mathrm{X}_\omega = \assympCone_\omega(X, d, (\lambda_n)_n, (\basePt_n)_n)$ and $\alpha_n \colon X \to X$ a sequence of isometries such that,
  \begin{gather}
  \sup_{n \in \N} \frac{1}{\lambda_n} d( \basePt_n, \alpha_n(\basePt_n) ) < \infty, \\
  \sup_{x_i \in X, n \in \N} d \left( \mu(\alpha_n(x_1), \alpha_n(x_2), \alpha_n(x_3)), \alpha_n(\mu(x_1,x_2,x_3)) \right) < \infty, \label{eq:condForInducingIsometricAutomorphisms}
 \end{gather}
 where $\mu$ is some representative of the coarse median structure.
 
 Then the map $\alpha_\omega \colon \mathrm{X}_\omega \to \mathrm{X}_\omega$ is an isometric automorphism of the metric median algebra $(\mathrm{X}_\omega, d_\omega, \mu_\omega)$.
\end{lem}

\subsection{Abstract version of Paulin's theorem} \label{subsec:genericPaulin}
\begin{sloppypar}
Let $G$ be a finitely generated group.
Consider an asymptotic cone $\assympCone_\omega(G, d_S, (\lambda_n)_n, (\basePt_n)_n)$ and a sequence of homomorphisms $\varphi_n \colon H \to G$ such that $\left( \frac{1}{\lambda_n} d(\varphi_n(h) \basePt_n, \basePt_n) \right)_n$ is bounded for all fixed $h \in H$.
Then by the discussion in \cref{subsec:asympCone}, we have an induced action of $H$ on $\assympCone(G, d_S, (\lambda_n)_n, (\basePt_n)_n)$ by isometries.
\end{sloppypar}
\begin{defi}
 We say that a group $H$ has property $\FHomCone(G)$ if for all asymptotic cones of $G$, every action constructed as above has a fixed point.
\end{defi}

We say two homomorphisms $\varphi_1, \varphi_2 \colon H \to G$ are \emph{conjugate} if there exists $g \in G$ such that $\varphi_2(h) = g \varphi_1(h) g^{-1}$ for all $h \in H$.

\begin{prop}[{Abstract version of Paulin's theorem~\cite{paulin:outerAutomorphismsOfHyperbolicGroups}}] \label{prop:genericPaulin}
 Let $G$ and $H$ be finitely generated groups such that $H$ has property $\fixedPtProperty_{\assympCone}(G)$.
 Then there are only finitely many conjugacy classes of homomorphisms from $H$ into $G$.
\end{prop}
\begin{proof}
 Fix a finite generating set $T \subseteq H$.
For each homomorphism $\varphi \colon H \to G$, consider the function $l_{\varphi} \colon G \to \Rgeq, l_\varphi(g) = \max_{t \in T} d(\varphi(t) g, g)$, where $d$ is some word metric on $G$ induced by a finite generating set.
Since $l_{\varphi}$ takes integer values, the minimum $\lambda_\varphi := \min_{g \in G} l_\varphi(g)$ always exists.
We note the following well-known fact.
\begin{lem}[{\cite{paulin:outerAutomorphismsOfHyperbolicGroups}}] \label{lem:paulinKeyLemma}
 If $\varphi_n \colon H \to G$ is a sequence of pairwise non-conjugate homomorphisms, then $\lambda_{\varphi_n} \to \infty$ as $n \to \infty$.
\end{lem}

Now suppose that \cref{prop:genericPaulin} is false, that is, there indeed exists a sequence of pairwise non-conjugate homomorphisms $\varphi_n \colon H \to G$.
Then for each $n \in \N$, choose a minimizing point $g_n \in G$ such that $l_{\varphi_n}(g_n) = \lambda_n := \lambda_{\varphi_n}$.
By \cref{lem:paulinKeyLemma}, we may construct an asymptotic cone $\mathrm{X}_\omega = \assympCone_\omega(G,d,(\lambda_n)_n, (g_n)_n)$.
By definition of $l_{\varphi_n}$, we have $d( \varphi_n(h) g_n, g_n) \leq |h|_T l_{\varphi_n}(g_n) = |h|_T \lambda_n$ for all $h \in H$.
In particular, the sequence of homomorphisms $(\varphi_n)_n$ induces an action of $H$ on $\mathrm{X}$, which, by assumption, has a fixed point $\mathrm{x} \in \mathrm{X}$.
The point $\mathrm{x}$ is represented by a sequence $(x_n)_n \subset G$.
The fact that $\mathrm{x}$ is a fixed point means that
\begin{equation*}
  \lim_\omega \frac{d\left(x_n, \varphi_n(h) x_n \right)}{\lambda_n} = 0, \qquad \forall h \in H.
 \end{equation*}
 In particular, since $T$ is finite, there exists $n_0 \in \N$ such that $d(x_{n_0}, \varphi_{n_0}(t) x_{n_0}) < \frac{\lambda_{n_0}}{2}$ for all $t \in T$, contradicting the definition of $\lambda_{n_0}$.
\end{proof}

We are now in a position to prove our main results.

\thmPaulinForCoarseMedian
\begin{proof}
 Let $G$ be a finitely generated group which admits a uniformly left-invariant coarse median structure and $H$ a finitely generated group with Kazhdan's property (T).
 It suffices to prove that $H$ has  $\fixedPtProperty_{\assympCone}(G)$, then the theorem follows from \cref{prop:genericPaulin}.

 Indeed, consider an asymptotic cone $\mathrm{X}_\omega = \assympCone_\omega(G, d_S, (\lambda_n)_n, (\basePt)_n)$ and a sequence of homomorphisms $\varphi_n \colon H \to G$ as in the beginning of this subsection.
 By \cref{prop:asymptoticConeIsMetricMedianAlgebra}, $\mathrm{X}_\omega$ admits the structure $(X_\omega, d_\omega, \mu_\omega)$ of a uniformly rectifiable metric median algebra of finite rank.
 For each fixed $h \in H$, let $\alpha_n(h) \colon G \to G$ denote the left multiplication map by $\varphi_n(h)$.
 Because the coarse median structure is uniformly left-invariant, the sequence of isometries $(\alpha_n(h))_n$ satisfies \labelcref{eq:condForInducingIsometricAutomorphisms}, whence \cref{lem:InducingIsometricAutomorphisms} says that the induced action of $H$ on $\mathrm{X}_\omega$ is by isometric automorphisms of the metric median algebra.
 Finally, \cref{thm:TfixedPoint} implies that the action has a fixed point.
\end{proof}
\corPaulinForCoarseMedian

\begin{ex} 
    Let $\Sigma$ be a compact orientable surface of genus $g$ and with $p$ punctures.
    Denote its mapping class group by $\mathrm{Map}(\Sigma)$.
    Bowditch~\cite[Sections 10--11]{bowditch:coarseMedian} has shown that the \emph{centroid} of Behrstock--Minsky~\cite{behrstock-minsky:centroidsAndRapidDecay} on $\mathrm{Map}(\Sigma)$ is a coarse median of rank at most $\xi(\Sigma)$, where $\xi(\Sigma) = 3g - 3 + p$ is the \emph{complexity} of $\Sigma$.
    Since the centroid is equivariant~\cite[Theorem 1.2]{behrstock-minsky:centroidsAndRapidDecay}, we observe that the induced coarse median structure is uniformly $\mathrm{Map}(\Sigma)$-invariant.
    Consequently, \cref{thm:paulinForCoarseMedian} is applicable to $\mathrm{Map}(\Sigma)$ and hence we recover~\cite[Theorem 1.2]{behrstock-drutu-sapir:MedianStructuresOnAymptoticCones}.
\end{ex}
\begin{ex} \label{ex:productsOfHyperbolicWithT}
 Let $G_i$ be a hyperbolic group with Kazhdan's property (T), where $i=1, \dotsc, n$.
 Such groups can be obtained as cocompact lattices in $\mathrm{Sp}(1,n)$ or constructed using the methods of Ballmann--{\'S}wi{\c{a}}tkowski~\cite{ballmann_swiatkowski:propertyT}.

 Then the direct product $G = G_1 \times \dotsm \times G_n$ has property (T) and admits a uniformly left-invariant coarse median structure of rank at most $n$. 
 We deduce from \cref{cor:paulinForCoarseMedian} that the outer automorphism group of $G$ is finite.
\end{ex}

\section{Metric approximation of finite subsets} \label{sec:approxFiniteSubsets}
Here we discuss the approximation of finite subsets in a coarse median space by finite median metric spaces (or equivalently, finite $\cat(0)$ cubical complexes with rescaled edge lengths).
We start with the following observations about quasi-morphisms into coarse median spaces.

\begin{lem} \label{lem:quasiMorphismParallelEdges}
 For every $L \geq 0$ there exist $\beta, \gamma \geq 0$ such that the following holds.
 
 If $\mu \colon X^3 \to X$ is a coarse median, $C$ a $\cat(0)$ cubical complex and $f \colon (C^{(0)}, \mu_{C^{(0)}}) \to (X,\mu)$ an $L$-quasi-morphism, then,
 \begin{enumerate}[(i)]
  \item \label{item:parallelEdgesQM} For any two parallel edges $e_1$ and $e_2$ with boundary vertices $e_i^\pm$, $i \in \{1,2\}$, we have,
 \begin{equation*}
  \frac{1}{\beta} d(f(e_1^+), f(e_1^-)) - \gamma \leq d(f(e_2^+), f(e_2^-)) \leq \beta d(f(e_1^+), f(e_1^-)) + \gamma.
 \end{equation*}
 
  \item \label{item:combinatorialGeodesicQM} If $x_0, \dotsc, x_m \in C^{(0)}$ is a combinatorial geodesic in $C^{(1)}$, then we have,
  \begin{equation*}
   d(f(x_i), f(x_{i+1})) \leq \beta \, d(f(x_0), f(x_n))) + \gamma
  \end{equation*}
  for all $i \in \{0, \dotsc, m-1\}$.
  \end{enumerate}
\end{lem}
\begin{proof}
\begin{sloppypar}
  \labelcref{item:parallelEdgesQM} If necessary reorient the edges, so that there exists a unique wall $W = \{h, h^c\} \in \walls(C^{(0)})$ such that $e_i^- \in h$ and $e_i^+ \in h^c$, $i \in \{1,2\}$.
  It follows that $e_2^\pm = \mu_{C^{(0)}}(e_2^-, e_2^+, e_1^\pm)$ because half spaces are convex.
  Thus, since $f$ is a quasi-morphism, we have $d(f(e_2^\pm), \mu(f(e_2^-), f(e_2^+), f(e_1^\pm))) \leq L$.
  We conclude that
  \begin{align*}
   d(f(e_2^+), f(e_2^-)) &\leq d\left( \mu(f(e_2^-), f(e_2^+), f(e_1^+))), \mu(f(e_2^-), f(e_2^+), f(e_1^-)) \right) + 2 L \\
   &\leq k \, d(f(e_1^+), f(e_1^-)) + h(0) + 2 L,
  \end{align*}
  where $k, h$ are parameters of $\mu_{C^{(0)}}$. By symmetry, this proves \labelcref{item:parallelEdgesQM}.
  \end{sloppypar}
  \labelcref{item:combinatorialGeodesicQM} The unique wall which separates $x_i$ from $x_{i+1}$ also separates $x_0$ from $x_m$, and we have $x_i = \mu_C(x_i, x_{i+1}, x_0)$ as well as $x_{i+1} = \mu_C(x_i, x_{i+1}, x_m)$. Now the proof proceeds as in \labelcref{item:parallelEdgesQM}.
\end{proof}

\thmFiniteSubsetApproxByCCC
\begin{proof}
 Let $h, k$ be parameters of a coarse median $\mu \colon X^3 \to X$. 
 Due to \cref{rem:WLOGIfDontCareAboutRank} there exists a finite median algebra $(M, \mu_M)$ and an $h(|A|)$-quasi-morphism $f \colon M \to X$ such that $A \subseteq f(M)$.
 If $\mu$ has rank at most $n$, we may assume that $M$ has rank at most $n+1$.
 By \cref{lem:finiteSubalgebra}, we can also assume that $|M| \leq 2^{2^{|A|}}$.

 The finite median algebra $M$ is the vertex set of the finite $\cat(0)$ cubical complex $C = \cube(M)$.
 For each $W \in \walls(M)$, choose an edge $e$ which intersects $W$ and define $l(W) := d(f(e_-), f(e_+))$.
 By \cref{lem:quasiMorphismParallelEdges}, this does not depend on the choice of the edge up to fixed multiplicative and additive errors depending on the cardinality of $A$.
 Use the chosen $l \colon \walls(M) \to \Rgr$ to define a metric $d_l$ according to \cref{subsec:constructingMetricsOnMedianAlgebras}.
 
 Now let $x_0, \dotsc, x_m \in C^{(0)} = M$ be a combinatorial geodesic in $C^{(1)}$. Then, by \cref{lem:quasiMorphismParallelEdges} \labelcref{item:parallelEdgesQM}, we can estimate
 \begin{align*}
  d(f(x_0), f(x_n)) &\leq \sum_{i = 0}^{m-1} d(f(x_i), f(x_{i+1})) \\
  &\leq \sum_{i = 0}^{m-1} \left( \beta \, d_l(x_i, x_{i+1}) + \gamma \right) = \beta \, d_l(x_0, x_m) + m \gamma  \\
  &\leq \beta \, d_l(x_0, x_m) + |M| \gamma,
 \end{align*}
 where $\beta, \gamma$ are the constants from \cref{lem:quasiMorphismParallelEdges}. On the other hand,
 \begin{align*}
  d_l(x_0, x_m) &= \sum_{i=0}^{m-1} d_l(x_i, x_{i+1}) \\
  &\leq \sum_{i=0}^{m-1} \beta \, \left( d(f(x_i), f(x_{i+1})) + \gamma \right) \\
  &\leq \sum_{i=0}^{m-1} \beta \, \left( \left( \beta \, d(f(x_0), f(x_n)) + \gamma \right) + \gamma \right) \\
  &= m \beta^2 \, d(f(x_0), f(x_m)) + m (\beta \gamma + \gamma) \\
  &\leq |M| \beta^2 \, d(f(x_0), f(x_m)) + |M|  (\beta \gamma + \gamma),
 \end{align*}
 where we have used both parts of \cref{lem:quasiMorphismParallelEdges}.
 Since $\beta, \gamma$ only depend on $L := h(|A|)$ and the parameters of $\mu$, and $|M| \leq 2^{2^{|A|}}$, these two estimates show that $f \colon M \to X$ is an $(\alpha(|A|), \varepsilon(|A|))$-quasi-isometric embedding for functions $\alpha, \varepsilon$ that only have dependencies as allowed for in the statement of the theorem.
 \end{proof}
 The above result can be applied to hyperbolic geodesic spaces (since they admit a rank $1$ coarse median).
 However, in this special case, the classical theorem on approximating trees~\cite[\S 6.2]{gromov:hyperbolicGroups} is stronger than the result proved above.
 Namely, in the classical result there are only explicit additive errors and no multiplicative errors.
 Moreover, by the classical result, every finite subset is contained in a finite geodesic tree, whereas \cref{thm:finiteSubsetApproxByCCC} only give us quasi-isometric approximations by vertex sets of $2$-dimensional $\cat(0)$ cubical complexes with modified edge lengths.

\printbibliography[heading=bibintoc]
\myAffiliation

\end{document}